\documentclass[11pt,a4paper]{amsart}

\usepackage{amsmath,amsfonts,amssymb,amsthm}
\usepackage{txfonts}
\usepackage{latexsym,bm,graphicx}
\usepackage{mathrsfs}
\usepackage{color}

\title[Boundary regularity of harmonic maps]{Boundary regularity of harmonic map heat flows from RCD spaces into $CAT(0)$ spaces}
\author{Zi-Ying Liu}
\address{Department of Mathematics\\  Sun Yat-sen University\\ Guangzhou 510275\\ \newline E-mail address: liuzy265@mail2.sysu.edu.cn}
\author{Hui-Chun Zhang}
\address{Department of Mathematics\\  Sun Yat-sen University\\ Guangzhou 510275\\ \newline E-mail address: zhanghc3@mail.sysu.edu.cn}
\author{Xi-Ping Zhu}
\address{Department of Mathematics\\  Sun Yat-sen University\\ Guangzhou 510275\\ \newline E-mail address: stszxp@mail.sysu.edu.cn}

 \newtheorem{theorem}{Theorem}[section]

\newtheorem{proposition}[theorem]{Proposition}
\newtheorem{lemma}[theorem]{Lemma}
\newtheorem{corollary}[theorem]{Corollary}

\theoremstyle{definition}
\theoremstyle{remark}

\newtheorem{defn}[theorem]{Definition}
\newtheorem{remark}[theorem]{Remark}

\numberwithin{equation}{section}

\newcommand{\ls}{\leqslant}
\newcommand{\gs}{\geqslant}

\newcommand{\ip}[2]{\left<{#1},{#2}\right>}

\newcommand{\meas}{\mathfrak{m}}
\newcommand{\dm}{{\rm d}\meas}

\begin{document}

%\today

\begin{abstract}
 In this paper, we study the boundary regularity of harmonic map heat flows from a bounded domain in a metric measure space with generalized Ricci curvature bounded from below into a non-positively curved metric space. % The Zygmund-type estimates in the present paper is new even for solutions of heat equations on Euclidean spaces.
 %\\[3pt]

%\noindent{\bf Key words}: Dirichlet heat kernel, boundary regularity, harmonic maps.\\
%\noindent{\bf MSC 2020}: 58E20, 53C23
\end{abstract}

\maketitle
%\tableofcontents
%\setcounter{tocdepth}{1}

\section{Introduction}

Over the past two to three decades, geometric analysis on nonsmooth metric spaces has drawn much research interest. In the paper, we are concerned with the boundary regularity of harmonic map heat flows from metric measure spaces with generalized Ricci lower bounds into non-positively curved metric spaces in the sense of Alexandrov. Recall in the smooth setting that a harmonic map heat flow between two smooth Riemannian manifolds is a gradient flow of Dirichlet energy. 

To introduce our results, we need some notations. Let $(X,d,\meas)$ be a metric measure space, $\Omega\subset X$ a bounded open subset, and let $(Y,d_Y)$ be another complete metric space. The space $L^2(\Omega,Y)$ denotes all of Borel measurable maps $u:\Omega\to Y$ such that its range is separable and that for some $P\in Y$ (hence for all $P\in Y$)  the function $d_Y(P,u(\cdot))$ is in $L^2(\Omega):=L^2(\Omega,\meas)$. To define a Dirichlet energy of $L^2(\Omega,Y)$ map, Korevaar and Schoen \cite{KS93} introduced an approach via the weak convergence of functionals. Independently, Jost \cite{Jost94,Jost95} provided an methed via the $\Gamma$-convergence. These approximations were later adapted and modified by Kuwae and Shioya \cite{KS03} and   Gigli and   Tyulenev \cite{GT21b} into the following form: for any map $u\in L^2(\Omega,Y)$,  the approximating energy of $u$ of scalar $r>0$ is defined by
$$   e_{2,r}[u](x):= \fint_{B_r(x)} \frac{d_Y^2(u(x), u(y))}{r^2}{\rm d}\meas(y)   $$
whenever $B_r(x)\subset \Omega$, otherwise, $  e_{2,r}[u](x)=0$. Here $\fint_Af{\rm d}\meas:=\frac{1}{\meas(A)}\int_Af{\rm d}\meas$ for any measurable set $A$ and any $f\in L^1(A)$. A map $u\in L^2(\Omega,Y)$ is called to be in $W^{1,2}(\Omega,Y)$ if there exists a function $e_u(x)\in L^1(\Omega)$, called the energy density of $u$, such that 
$$\lim_{r\to0^+} e_{2,r}[u]=e_u\ \ {\rm in}\ \ L^1(\Omega).$$ 
In such a case, the (Dirichlet) energy of $u$ is defined by 
\begin{equation*}\label{equ:Dirichlet-energy-of-map}
E[u]:=\int_\Omega e_u(x)\dm(x).	
\end{equation*}
Remark that some other notions of energy of a map $u\in L^2(\Omega,Y)$ have been well-developed in \cite{Jost97, Jost98,HKST01,Ohta04,Guo21}.

  Mayer \cite{May98} and J. Jost \cite{Jost98} independently extended Crandall-Liggett's theory of gradient flows from convex functionals on Hilbert spaces to convex functionals on non-positively curved metric spaces (denoted $CAT(0)$ spaces).  
Based on this abstract theory, a weak solution of harmonic map heat flows between nonsmooth metric spaces can be defined as a gradient flow of the energy $E[u]$ on $L^2(\Omega,Y)$. Under the assumption that $Y$ is a $CAT(0)$ space, they established favorable long-time existence, uniqueness, and well-established long-time behaviors  (see \cite{May98,Jost98} and Guo \cite{Guo21}).

The natural question is to consider the regularity of such weak solutions of harmonic map heat flows. Let $(X,d,\meas)$ be an $RCD(K,N)$ space, $K\in\mathbb R$ and $N\in(0,+\infty)$, and let $(Y,d_Y)$ be a $CAT(0)$ space. Recently, Han and the authors \cite{ZZ26a, HZZ26} have obtained interior Lipschitz continuity for any harmonic map heat flow from a bounded domain $\Omega\subset X$ into $Y$. In the case where the source space $X$ is a smooth manifold, the same local Lipschitz regularity result has been established by Lin-Seggati-Sire-Wang \cite{LSSW26} and Lin-Wang \cite{LW26} using a different approach.

In the present paper, we focus on the boundary regularity of harmonic map heat flows from an $RCD(K,N)$ domain into $CAT(0)$ spaces. Precisely, letting $\Omega\subset X$ be a bounded domain of an $RCD(K,N)$ space, $K\in \mathbb R$, $N\in(1,+\infty)$, and $(Y,d_Y)$ a $CAT(0)$ space, given
 any $\psi\in W^{1,2}(\Omega,Y)$,  we put
$$W^{1,2}_\psi(\Omega,Y):=\Big\{u\in W^{1,2}(\Omega,Y)|\ d_Y\big(u(x),\psi(x)\big)\in W^{1,2}_0(\Omega),\ \ E[u]\ls E[\psi]\Big\}.$$ 
Let $u(x,t):\Omega\times(0,+\infty)\to Y$ be the harmonic map heat flow with the initial-boundary data $(\psi,u_0)$, i.e., the gradient flow of $u\mapsto E[u]$ starting from some $u(\cdot,0)=u_0$ in $W^{1,2}_\psi(\Omega,Y)$. We consider the qualitative behavior of $u(x,t)$ as $x\to \partial\Omega$.

 We need the following four conditions for the boundary.
 \begin{defn} \label{BRHMHF-def-1.5}
 (1) $\partial\Omega$ is said to satisfy the \emph{Wiener criterion} if for each $x_0\in \partial\Omega$ it holds
 \begin{equation*}\label{equ:Wiener-criterion}
\int_0^1\frac{{\rm cap}(B_r(x_0)\setminus \Omega, B_{2r}(x_0))}{{\rm cap}(B_r(x_0), B_{2r}(x_0))}\frac{{\rm d}r}{r}=+\infty.
\end{equation*} 
 (2)  $\Omega$ is said to satisfy a  {\emph{uniformly exterior density condition with radius $R_{\rm den}\in(0,1)$}} if there exists  $c_0>0$ such that
\begin{equation*} 
\frac{\meas(B_r(x_0)\setminus\Omega)}{\meas(B_r(x_0))}\gs c_0,\quad \forall x_0\in\partial\Omega,\ \ \forall r\in(0,R_{\rm den}).
\end{equation*}

(3)  $\Omega$ is said to satisfy a  {\emph{uniformly exterior ball condition with radius $R_{\rm ext}\in(0,1)$}} if for each $x_0\in \partial\Omega$ there exists a point $y_0\in X\setminus\Omega$ such that
\begin{equation*}\label{BRHMHF-equ-1.2}
d(y_0,x_0)=R_{\rm ext}\quad {\rm and}\quad \overline{B_{R_{\rm ext}}(y_0)}\cap \Omega=\emptyset.
\end{equation*}

(4)   $\partial\Omega$ is said to be  {\emph{perimetrically regular}}  if it is a set of finite perimeter and  there exists a constant  $C_0>0$ such that
\begin{equation*}\label{BRHMHF-equ-1-3}
Per_{\Omega}(B_r(x_0))\ls C_0\frac{\meas(B_r(x_0))}{r},\qquad \forall x_0\in\partial\Omega,\ \ \forall r\in (0,{\rm diam}(\Omega)).
\end{equation*}
See the definition of the capacity ${\rm cap}$ in Section \ref{sect3} and the perimeter $Per_\Omega$ in Section \ref{sect-5}.
\end{defn}

A map $u\in C^\alpha(\overline\Omega,Y)$, $\alpha\in(0,1)$, means 
$$ d_Y(u(x),u(y))\ls Cd^\alpha(x,y),\quad \forall x,y\in\overline\Omega,$$
for some $C>0$. If this inequality holds for $\alpha=1$, we will write $u\in Lip(\overline\Omega,Y)$.

\begin{theorem}\label{thm:main-BRHMHF}
Let $\Omega$ be a bounded domain in an $RCD(K,N)$ space $(X,d,\meas)$, $K\in \mathbb R$, $N\in(1,+\infty)$, and let $(Y,d_Y)$ be a $CAT(0)$ space. Given any $\psi\in W^{1,2}(\Omega,Y)$, let  $u(x,t):\Omega\times(0,+\infty)\to Y$ be a weak solution of the harmonic map heat flow in $W^{1,2}_{\psi}(\Omega,Y)$,  with initial value $u_0$. Suppose that the image of $u_0$ is bounded in $Y$. Then the following statements hold.
 \begin{itemize}
 \item[(1)] If $\partial\Omega$ satisfies the Wiener criterrion, and   $\psi\in C(\overline\Omega,Y)$, then $u(x,t)$ is in $C(\overline\Omega\times(0,+\infty),Y)$.
 \item[(2)] If $\Omega$ has a uniformly exerior density condition, and   $\psi\in C^\alpha(\overline\Omega,Y)$ for some $\alpha\in(0,1)$, then $u(x,t)$ is in $C^\beta(\overline\Omega\times(0,+\infty),Y)$ for some $\beta\in(0,\alpha)$.
\item[(3)] If $\Omega$ has a uniformly exerior ball condition, and   $\psi\in Lip(\overline\Omega,Y)$, then $u(x,t)$ is in $C^{1-\epsilon}(\overline\Omega\times(0,+\infty),Y)$ for any $\epsilon\in(0,1)$.
\item[(4)] If $\Omega$ has a uniformly exerior ball condition with radius $R_{\rm ext}$ and the boundary $\partial\Omega$ is perimetrically regular, and   $\psi\in Lip(\overline\Omega,Y)$ with a Lipschitz constant $L_\psi$.  Let $x_0\in \partial\Omega$ and $R:=\min\{{\rm diam}(\Omega)/100,  R^2_{\rm ext}/(4L_\psi)\}$. Then for any $T\gs R^2_{\rm ext}$ the following Zygmund-type estimate holds
   
  \begin{equation}\label{equ-Zygmund-1}
  	\frac{d_Y(u(x,T+t),\psi(x_0)) }{d(x,\partial\Omega)}\ls c_T \ln\left(\frac{e\cdot {\rm diam}(\Omega)}{d(x,\partial\Omega)}\right)
  \end{equation} 
for any $t>0$ and any $x\in\Omega\cap B_{R/4}(x_0)$ with $d(x,x_0)\ls 2 d(x,\partial\Omega)$, where  the constant $c_T$ depends only on $N,K,\Omega,M, L_\psi$ and $T$.
 \end{itemize}
 \end{theorem}

Considering the smooth setting, the Wiener criterion has been extended to harmonic maps from bounded Lipschitz domains in $\mathbb R^n$, whose image lies in a convex ball \cite{ACM91}. Jost-Meier \cite{Jost-M83} has established the H\"older boundary regularity for minima of certain quadratic functionals, assuming $C^1$-continuity of the boundary $\partial \Omega$ for $\Omega\subset\mathbb R^n$.  
 Serbinowski \cite{Ser94}  proved that any harmonic map into a $CAT(0)$ space is globally H\"older continuous, assuming that the boundary of $\Omega\subset\mathbb R^n$ is $C^1$ and that the data is Lipschitz continuous.

\begin{remark}\label{rem:boundary-regularity-hmhf}
 (1)   The exterior ball condition in Theorem \ref{thm:main-BRHMHF} (3) is nearly indispensable, as illustrated in the following simple examples. Let $a\in(1/2,1)$ and consider the cone
 $$C_a:=\{(\rho,\theta) |\ 0<\rho<+\infty,\ 0<\theta<\pi/a\} $$
with  the vertex $O=(0,0)$, where $(\rho, \theta)$ are the polar coordinates in $\mathbb R^2$. Then the domain $\Omega_a:=C_a\cap B_1(O)$ satisfies only an exterior cone condition at  $O$ with the cone angle $(2-1/a)\pi<\pi$. A direct calculation shows that the function $f_a:=\rho^{a}\sin(a\theta)$ is harmonic on $\Omega_a$ and has Lipschitz boundary data (indeed $f_a(x)=0$ for any $x\in\partial C_a\cap B_1(O)$). Obviously, it is not Lipschitz near the vertex $O$. 

(2)  The Zygmund-type estimate in (\ref{equ-Zygmund-1}) is optimal, even for a smooth domain in the Euclidean setting $X=\mathbb R^n$. This can be understood by the following Hardy-Littlewood's example.  
Let  $B^2_{1}(0)=\{(r,\theta)|\ 0\ls r<1, \theta\in(-\pi,\pi]  \}$ be the unit ball in $ \mathbb R^2$, and let $f_0$ be defined on $S_1=\partial B^2_1(0)$ by
$$f_0(\theta)=|\sin\theta|  \  \ {\rm if}\ \  |\theta|\ls \pi/2 \qquad {\rm and}\qquad   f_0(\theta)=1\ \ {\rm otherwise}.$$
Then the harmonic function (a static solution of the heat equation) on $B^2_{1}(0)$ with the data  $f_0$ is uniquely determined by the Poisson integral formula,
$$f(r,\theta_0)= \frac{1-r^2}{2\pi}\int_{ S^1} \frac{ f_0(\theta)}{|(r,\theta_0),(1,\theta)|_{\mathbb R^2}^2 }{\rm d}\theta = \frac{1-r^2}{2\pi}\int_{S^1} \frac{f_0(\theta)}{ r^2+1-2r\cos(\theta-\theta_0)   }{\rm d}\theta.$$
 By direct calculation, and noticing $ d((r,0),\partial B_1^2(0))=1-r$, we have
 \begin{equation*}
 \begin{split}
2\pi\cdot  f(r, 0) &= (1-r^2)\int_{-\pi}^\pi\frac{ |\sin\theta| }{r^2+1-2r\cos \theta }{\rm d}\theta=2(1-r^2)\int_{0}^\pi\frac{  \sin\theta  }{r^2+1-2r\cos \theta }{\rm d}\theta\\
   & =2\frac{(1-r^2)}{r}\ln\left(\frac{1+r}{1-r}\right)\gs (1-r)\ln(\frac{1}{1-r}).
    \end{split}
 \end{equation*}

\end{remark}

 We will organize this paper as follows.   In Section \ref{sect2}, we will provide some necessary material on  $RCD(K, N)$ spaces. In Section \ref{sect3}, we will consider Wiener's criterion for parabolic equations on $RCD(K,N)$ spaces. In Section \ref{sect4}, we will give the optimal H\"older boundary behaviors of nonnegative subsolutions to the heat equations, assuming an exterior ball condition. A Zygmund-type estimate for the subsolutions of heat equations will be established in Section \ref{sect-5}. Section \ref{sect-6} is the proof of the main Theorem \ref{thm:main-BRHMHF}.  \\

{\bf Acknowledgements} The third author was partially supported by National Key R\&D Program of China (No. 2022YFA1005400) and NSFC 12271530. The second author was partially supported by NSFC 12426202.

\section{Preliminaries and notations}\label{sect2}

 Let $(X,d,\meas)$ be a metric measure space, i.e., $(X,d)$ is a complete and separable metric space equipped with a non-negative Borel measure which is finite on any ball $B_r(x)\subset X$ and ${\rm supp}(\meas)=X$.
  Given an open domain $\Omega\subset X$, we denote by $Lip(\Omega)$ (resp.  $C_0(\Omega)$, $Lip_0(\Omega)$, $Lip_{\rm loc}(\Omega)$) the space of Lipschitz continuous (resp. continuous with compact support, Lipschitz continuous with compact support, locally Lipschitz continuous) functions on $\Omega$.
  We denote by $L^p(A):=L^p(A, \meas)$ for any $p\in[1,+\infty]$ and any $\meas$-measurable subset $A\subset X$. Recall that, for any $f\in L^1(A)$, we set
\begin{equation*}\label{equ:avarage-functions}
	f_A:=\fint_Af{\rm d}\meas:=\frac{1}{\meas(A)}\int_Af{\rm d}\meas.
\end{equation*}

 \subsection{Basic calculus on $RCD(K,N)$ spaces}

 Let $f\in C(X)$, the \emph{pointwise Lipschitz constant} or \emph{slope}, ${\rm Lip}f: X\to [0,+\infty]$, is given by
 \begin{equation}\label{equ:BRHMHF-2.1}
 {\rm Lip}f(x):= \limsup_{y\to x}\frac{|f(y)-f(x)|}{d(x,y)} 
     \end{equation}
 if $x$ is not isolated, and ${\rm Lip}f(x)=0$ otherwise.
 The \emph{Cheeger energy} ${\rm Ch}: L^2(X)\to [0,+\infty]$ is
 $${\rm Ch}(f):=\inf\Big\{\liminf_{j\to +\infty}\int_X({\rm Lip}f_j)^2{\rm d}\meas:\ f_j\in Lip_{\rm loc}(X)\cap L^2(X),\ \ f_j\overset{L^2}{\to }f\Big\}.$$
The Sobolev space
$$W^{1,2}(X)=W^{1,2}(X,d,\meas):=\{f\in L^2(X)|\ {\rm Ch}(f)<\infty\}$$
with the norm
\begin{equation*}
    \|f\|^2_{W^{1,2}(X)}:=\|f\|^2_{L^2(X)}+{\rm Ch}(f).
\end{equation*}

A metric measure space $(X,d,\meas)$ is called \emph{infinitesimally Hilbertian} if the Sobolev space $W^{1,2}(X)$ is a Hilbert space.  For an infinitesimally Hilbertian space $(X,d,\meas)$, it was proved \cite{Gig15,GP20} that  for any $f,g\in W^{1,2}(X)$ the limit
  \begin{equation}\label{BRHMHF-equ-2.2}
  \ip{\nabla f}{\nabla g}:=\lim_{\epsilon\to0^+}\frac{|\nabla (f+\epsilon g)|^2-|\nabla f|^2}{2\epsilon}
  \end{equation}
    exists and is in $L^1(X)$. 
    
 Let $K\in \mathbb R$ and $N\in[1,+\infty]$,  several equivalent definitions exist for the \emph{Riemannian curvature-dimension condition}, denoted by $RCD(K,N)$, on metric measure spaces \cite{EKS15, AMS16, AGS14b}. An $RCD(K,N)$ space is a generalized notion of $X$ having Ricci curvature bounded from below by $K$ and the dimension bounded from above by $N$. We refer the readers to \cite{EKS15,AMS16, AGS14b,Amb18} for the precise definition of the RCD condition. Here, we list some necessary properties for our purpose in the present paper.

From now on, we assume always that $(X,d,\meas)$ is an $RCD(K,N)$ space for some $K\in\mathbb R$ and $N\in [1,+\infty)$. Since an $RCD(K,N)$ space must be an $RCD(K',N')$ space for any $K'\ls K$ and $N'\gs N$. Without loss of generality, the parameters $K, N$ are always assumed to be $K\ls0$ and $N>1$.
  It holds on an $RCD(K,N)$ space $(X,d,\meas)$:\\
$\bullet$ (Local volume doubling) Let $R>0$. There exists a constant $C_D>0$, depending only on $N,K,R$ such that 
 \begin{equation*}\label{BRHMHF-equ-2.4}
 	\frac{\meas(B_{s}(x))}{\meas(B_{r}(x))}\ls C_D \cdot\left(\frac{s}{ r}\right)^N,\quad \forall\ 0<r<s<R,\ \ \forall \ x\in X.
 	 \end{equation*}
$\bullet$ (Locally weak (1,1)-Poincar\'e inequality) There exists a constant $C_P>0$, depending only on $N,K,R$ such that 
\begin{equation*}\label{equ-add-Poincare-ineq}
	\fint_{B_r(x)}|f-f_{B}|{\rm d}\meas\ls C_P \cdot r\fint_{B_{2r}(x)}|\nabla f|{\rm d}\meas,\quad \forall 0<r<R,\ \ \forall \ x\in X,
\end{equation*}
for any locally Lipschitz function $f$ on $X$ (see \cite{Raj12}).\\
$\bullet$  For any $f\in W^{1,2}(X)$,  there exists a  {\it minimal weak upper gradient} $|\nabla f|$ such that
 $${\rm Ch}(f)=\frac{1}{2}\int_X|\nabla f|^2(x){\rm d}\meas(x),$$
and for any $f\in Lip(X) $ that $|\nabla f|(x)={\rm Lip}f(x)$  $\meas-$a.e. (see \cite{Che99,AGS14a}).

 Let $\Omega\subset X$ be an open set. A  function $f\in L^2_{\rm loc}( \Omega)$ belongs to
$W^{1,2}_{\rm loc}(\Omega)$, provided, for any    $\chi\in Lip_0(\Omega)$   it holds
$f\chi\in W^{1,2}(X)$, where $f\chi$ is understood to be  $0$ outside of $\Omega$.
 The space
$$W^{1,2}(\Omega):=\big\{f\in W^{1,2}_{\rm loc}(\Omega): \ {\rm both}\  f\ {\rm and}\  |\nabla f|\in L^2(\Omega)\big\},$$
and the space $W_0^{1,2}(\Omega)$ is defined as the $W^{1,2}(X)$-closure of the space of functions $f\in Lip_0(\Omega)$, where $f\in Lip_0(\Omega)$ is understood to be $0$ outside of $\Omega$.

When $\Omega$ is bounded and ${\rm diam}(\Omega)< {\rm diam}(X)$, the canonical Dirichlet form, $(\mathscr E_{\Omega},W^{1,2}_0(\Omega))$, is given by
\begin{equation*}
\mathscr E_\Omega(f):=\int_\Omega|\nabla f|^2{\rm d}\meas,\qquad f\in W^{1,2}_0(\Omega).
\end{equation*}
This canonical Dirichlet form is strongly local and regular   (see, for example, the proof of \cite[Lemma 6.7]{AGS14b} and \cite{Stu95}).
  Its associated infinitesimal generator is $\Delta_\Omega$ with domain $D(\Delta_\Omega)$, and the associated analytic semi-group is $(H^\Omega_tf)_{t\gs0}$ for any $f\in L^2(\Omega)$, which is given by 
 $$ H^\Omega_tf(x)=\int_\Omega p^\Omega_t(x,y)f(y){\rm d}\meas,\quad t\geqslant 0,\ \ \forall f\in L^2(\Omega),$$
  where $p^{\Omega}_t(x,y)$ is the Dirichlet heat kernel on $\Omega$. Assuming that $\Omega$ satisfies a uniformly exterior ball condition with radius $R_{\rm ext}\in(0,{\rm diam}(\Omega)/100)$,
an upper bound has been proved in \cite[Theorem 1.10]{ZZ25}  for any $T\gs R^2_{\rm ext}$ that
\begin{equation}
	\label{BRHMHF-equ-1.6}
	p^\Omega_t(x,y)\ls \frac{d(x,\partial\Omega)\cdot d(y,\partial\Omega)}{t}\cdot\frac{ c_{T}}{\meas(B_{\sqrt t}(y))} \exp\left(-\frac{d^2(x,y)}{c_T \cdot t}  \right), \quad \end{equation}
for any $t\in(0,T)$ and any $x,y\in \Omega$, where the constant $c_T $ depends only on $N,K,\Omega$ and $T$.

 Let us recall the definition of the  \emph{measure-valued Laplacian}.
 \begin{defn}[measure-valued Laplacian]\label{BRHMHF-def-2.4}
 Given a function $f\in W^{1,2}_{\rm loc}(\Omega)$,  its measure-valued  Laplacian $\mathbf\Delta  f$ is defined as a (linear) functional
\begin{equation*}
{\bf \Delta} f(\phi):=-\int_\Omega \ip{\nabla f}{\nabla \phi} {\rm d}\meas,\qquad \forall \phi\in Lip_0(\Omega).
\end{equation*}
 \end{defn}
When  $f\in W^{1,2}(\Omega)$, $\mathbf\Delta f$ can be extended to a functional on $W^{1,2}_0(\Omega).$ When $f\in W^{1,2}_0(\Omega)$, it was proved \cite{Gig15} that if there is   $g\in L^2(\Omega)$ such that $\mathbf \Delta f= g$ in the sense of distributions  then $f\in D(\Delta_{\Omega})$ and $\Delta_\Omega f=g.$
 Conversely, it is clear that if $f\in D(\Delta_{\Omega})$ and $\Delta_\Omega f=g,$ then $\mathbf \Delta f= g$ in the sense of distributions.

Given $f\in W^{1,2}_{\rm loc}(\Omega)$ and a signed Radon measure $\mu$, the notion ``${\bf \Delta} f \gs \mu$ in the sense of distributions" means that
 $$-\int_\Omega \ip{\nabla f}{\nabla \phi} {\rm d}\meas\gs \int_\Omega  \phi {\rm d}\mu,\qquad \forall\ \phi\in Lip_0(\Omega),\ \phi\gs0.$$	
In this case, the functional ${\bf \Delta}f$ provides a signed Radon measure on $\Omega$, denoted by ${\bf \Delta}f$ again  (by the Riesz representation theorem). Therefore, in this case, we write also  ``${\bf \Delta}f\gs \mu$ as measures". If the measure $\mu=g\cdot\meas$ for some $g\in L^1_{\rm loc}(\Omega)$, we  denote by ``${\bf \Delta} f \gs g$ in the sense of distributions" too.

\begin{defn}
	\label{BRHMHF-def-3.1}
Let $0<T\ls\infty$ and let $\Omega_T:=\Omega\times(0,T)$. A function $u\in L^2_{\rm loc}\big((0,T); W^{1,2}_{\rm loc}(\Omega)\big)$ is called a {\emph{super-solution of the heat equation}}, if for every open set $U\Subset\Omega_T$ and for all non-negative function $\phi(x,t)\in Lip(\Omega_T)$ with ${\rm supp}(\phi)\subset U$, we have
$$\int_U u\frac{\partial\phi}{\partial t}{\rm d}\meas {\rm d}t\ls \int_U\ip{\nabla u}{\nabla \phi}
	{\rm d}\meas {\rm d}t.$$
A function $u$ is called a {\emph{sub-solution of the heat equation}} if $-u$ is a super-solution of the heat equation.
\end{defn}

 We need the parabolic comparison principle given in \cite[Theorem 4.1]{KM15-pams}.
\begin{lemma}\label{BRHMHF-lem-3.2}
	Let $u\in L^2\big((0,T); W^{1,2}(\Omega)\big)$ be a super-solution of heat equation and let $v\in L^2\big((0,T); W^{1,2}(\Omega)\big)$ be a sub-solution of heat equation. Suppose $u\gs v$ near the parabolic boundary of $\Omega_T$ in the sense that for almost every $0<t<T$ we have the lateral boundary condition
	$$\big(v(\cdot,t)-u(\cdot,t)\big)_+\in W^{1,2}_0(\Omega)$$ and also
	the initial condition
	$$\frac 1 \varepsilon\int_0^\varepsilon\int_\Omega(v-u)_+^2{\rm d}\meas {\rm d}t\to 0,\quad {\rm as}\ \ \varepsilon\to0^+.$$
		Then
		$$u\gs v,\quad \meas\times \mathscr L^1{\rm-a.e.\  in}\ \  \Omega_T,$$ where $\mathscr L^1$ is the 1-dimensional Lebesgue measure on $(0,T)$.
		\end{lemma}

\section{The Wiener criterion and H\"older boundary continuity}\label{sect3}

Let $(X,d,\meas)$ be an $RCD(K,N)$ space with $K\ls0$ and $N\in(1,+\infty)$.  We fix a bounded open subset  $\Omega$  with $\meas(\partial\Omega)=0$ and ${\rm diam}(\Omega)< {\rm diam}(X)$.

In this section, we will use $c_1, c_2, C_1, \cdots$ to denote  generic  positive constants depending only on $K$, $N$, $ {\rm diam}(\Omega), \meas(\Omega)$ and $R_{\rm den}$ in Definition \ref{BRHMHF-def-1.5}; they may differ from line to line.

 For any ball 
 $B_r(x)\subset X$  and a measurable set $E\subset B_r(x)$, the \emph{capacity} is defined  \cite{KKM00} by
$${\rm cap}(E, B_{2r}(x)):=\inf\left\{\int_{B_{2r}(x)}|\nabla f|^2{\rm d}\meas  \big|\ f\in W^{1,2}_0(B_{2r}(x))\ {\rm such \ that }\ f\gs 1\ {\rm on   }\ E \right\}. $$

Recall the following oscillation estimates in \cite{Bjo02}, (see also \cite[Theorem 11.23]{BB-book} and \cite[Equ. (8.81)]{GT01}).
\begin{lemma}\label{lem:boj-boundary-estimate}
Let $U\subset \Omega$ be an open subset, and let $w\in W^{1,2}(X)\cap C(\overline U)$ and $f\in L^q(U)$ for some $q>\max\{N/2,3/2\}$. Suppose that  $u\in W^{1,2}(U)$ satifies 
${\bf\Delta}u=f$ on $U$ in the sense of distributions with the boundary condition $u-w\in W_0^{1,2}(U)$. Let $x_0\in \partial U$ and 
   $0<R< {\rm diam}(\Omega)/100$. Then  
$${\rm osc}_{U_R}u\ls \left(1-C\lambda_{x_0}(R)\right)\cdot {\rm osc}_{U_{4R}}u+C\lambda_{x_0}(R)\cdot {\rm osc}_{U_{4R}}w+CR^{2-N/q}\|f\|_{L^q(U)},$$
 where $U_R:=U\cap B_R(x_0)$ and 
\begin{equation}\label{equ:wiener-criterion-1}
	\lambda_{x_0}(R):=\frac{{\rm cap}(B_R(x_0)\setminus U, B_{2R}(x_0))}{{\rm cap}(B_R(x_0), B_{2R}(x_0))}. \end{equation}
 \end{lemma}

Consequently, the following regularity holds.

\begin{corollary}\label{lem:Wiener-criterion}
Let $U\subset \Omega$ be an open subset such that $\partial U$ satisfies the Wiener criterion
$\sum_{j\gs1}\lambda_{x_0}(4^{-j})=+\infty$ for any $x_0\in \partial U$, where $\lambda_{x_0}(R)$ is given in (\ref{equ:wiener-criterion-1}). Let $w\in W^{1,2}(X)\cap C(\overline U)$. Suppose that  $u\in W^{1,2}(U)$ satifies 
$${\bf\Delta}u=f,\quad u-w\in W_0^{1,2}(U)$$
for some $f\in L^q(U)$ in the sense of distributions, $q>\max\{N/2,3/2\}$. Then $u\in C(\overline U)$.
 \end{corollary}

Similarly, if the boundary $\partial U$ satisfies an exterior density condition and the boundary data $w$ is H\"older continuous near $\partial U$, then the solution of Possion equation is H\"older continuous in $\overline U$. The precise statement is the following.

\begin{corollary}\label{lem:Holder-continuous-for-torsion}
Let $U\subset \Omega$ satisfy  a uniformly exterior density condition, and   $w\in C^\alpha(\overline U)$ for some $\alpha\in(0,1)$. Suppose that  $u\in W^{1,2}(U)$ satifies 
$${\bf\Delta}u=f,\quad u-w\in W_0^{1,2}(U)$$
for some $f\in L^q(U)$ in the sense of distributions, $q>\max\{N/2,3/2\}$. Then $u \in C^\gamma(\overline U)$ for some $\gamma\in (0,\alpha).$ 
 \end{corollary}
\begin{proof}
It suffices to check that $\lambda_{x_0}$ is bounded from below by a positive constant $c>0$ uniformly with respect to $x_0\in\partial U$ and $ R\in (0,R_0)$. In fact, from \cite[Lemma 3.3]{Bjo02} and the exterior density condition in Definition \ref{BRHMHF-def-1.5}(1), we have
$$\lambda_{x_0}(R)\gs C\cdot\frac{\meas(B_R(x_0)\setminus U)}{CR^2}\cdot \frac{R^2}{\meas(B_R(x_0))}\gs c$$for all $R\in(0,R_{\rm den}).$
\end{proof}

To construct a barrier function near the boundary,  the following fact, on the existence of a neighborhood for a given point such that it has a good boundary, is very useful.
\begin{lemma}
	\label{lem:good-neighborhood}
	Let $x_0\in X$ and $R\ls {\rm diam}(\Omega)/100$.  Then there exists a domain $\widetilde B_R$ with 
	$$B_{R}(x_0)\subset \widetilde B_R\subset B_{11R/10}(x_0)$$
	such that  $\partial \widetilde B_R$ has an exterior ball condition with the exterior ball of radius $R_{{\rm ext}}\gs R/20.$
	\end{lemma}
\begin{proof}
	Define
	$$\widetilde B_R:=\{x\in B_{11R/10}(x_0): \ d(x,\partial B_{11R/10}(x_0))>R/10\}.$$
	The triangle inequality ensures $B_R(x_0)\subset \widetilde B_R$. Take any $x\in \partial \widetilde B_R$. There exists a nearest point $y_x\in \partial B_{11R/10}(x_0)$ such that $d(x,y_x)=R/10=d(x,\partial B_{11R/10}(x_0))$, since $\partial B_{11R/10}(x_0)\not=\emptyset.$ Let $\gamma_{x,y_x}$ be a shortest curve connected $x$ and $y_x$. 
	
	For any $r\in(0,R/20)$, we can find a point $z$ in $\gamma_{x,y_x}$ such that $d(x,z)=r$, and then $d(z,y_x)=R/10-r$. By using the triangle inequality again, we know that 
	$$ B_{r}(z)\cap B_R(x_0)=\emptyset.$$ 
Indeed, if there exists a point $p\in B_r(z)\cap B_R(x_0)$, then we have a contradiction 
$$11R/10=d(x_0,y_x)\ls d(x_0,p)+d(p,z)+d(z,y_x)<R+r+(R/10-r)=11R/10.$$
The proof is finished. 
\end{proof}

Let $U_1$ and $U_2$ be two bounded open sets in $X$ with $U_1\cap U_2\not=\emptyset$. From the definition of the Wiener criterion, it is clear that if their boundaries satisfy the Wiener criterion, then the boundary of $U_1\cap U_2$ also satisfies the Wiener criterion.  The same statement also holds by replacing the Wiener criterion with the exterior density condition, and also for the case of the exterior ball condition.  

Applying the above regularity to the torsion function (i.e., a solution of the Poisson equation ${\bf \Delta}h=-1$ in the sense of distributions), one can get the following constructions of continuous barriers.
\begin{lemma}
	\label{lem:barrier-elliptic}
		Let $\Omega\subset X$ be a bounded open domain satisfying the Wiener criterion, and let $w\in W^{1,2}(X)\cap C(\overline\Omega)$. Then for any  $x_0\in \partial\Omega$ and $R\in(0,{\rm diam}(\Omega)/100)$, there exist a domain $V_R(x_0)$ with 
		$$\Omega\cap B_R(x_0)\subset V_R(x_0)\subset \Omega \cap B_{11R/10}(x_0),$$ 
		and a function $h\in C(\overline{V_R(x_0)})\cap W^{1,2}(X)$ such that 
		$$  {\bf \Delta}h=-1 \ \ {\rm on}\ \ V_R(x_0) $$
		in the sense of distributions, and 
		$$h|_{\partial V_R(x_0)}\gs |w|,\qquad h|_{\partial\Omega\cap B_{R/2}(x_0)}=|w|,\qquad h|_{\partial V_R(x_0)\cap \Omega}= \max_{\overline V_R(x_0)}|w|.$$
\end{lemma}
\begin{proof} Let $\widetilde B_R$ be given in Lemma \ref{lem:good-neighborhood}, and set 
$$V_R(x_0):=\Omega\cap \widetilde B_R.$$
Let $w_1(x)$ be a continous function on $\overline{V_R(x_0)}$ such that $w_1=0$ on $\Omega\cap B_{R/2}(x_0)$ and $w_1=\max_{\overline{V_R(x_0)}}|w|$ on $\partial V_R(x_0)\cap\Omega$. Such a continuous function exists because 
$$d\big(B_{R/2}(x_0)\cap\Omega, \partial V_R(x_0)\cap \Omega\big)\gs d\big(B_{R/2}(x_0), \partial \widetilde B_R\big) \gs R/2.$$
	We define a function 
	$$\overline w(x):=\max\big\{ w_1(x),\ |w(x)|\big\},\qquad \forall x\in \overline{V_R(x_0)}.$$
It is clear that  $w\in C(\overline{V_R(x_0)})$, $\overline w(x)\gs |w(x)|$ for all  $x\in \overline{V_R(x_0)}$, $\overline w(x)= |w(x)|$ for all  $x\in \partial\Omega\cap B_{R/2}(x_0)$, 
and $\overline w(x)=\max_{\overline{V_R(x_0)}}|w|$ for all $x\in \partial V_R(x_0)\cap \Omega.$
 Now, the solution of the Dirichlet problem 
	$${\bf\Delta} h=-1\ \ \ {\rm on}\ \ V_R(x_0),\quad h-\overline w\in W^{1,2}_0(V_R(x_0))$$
	meets all of the assertions, by using Corollary \ref{lem:Wiener-criterion} and $V_R(x_0)\supset B_R(x_0)\cap \Omega$. The continuity of $h$ on $\overline{V_R(x_0)}$ is ensured by Corollary \ref{lem:Wiener-criterion}.
		\end{proof}

By using the same argument and replacing the boundary regularity of the torsion function $h$ (replacing Corollary \ref{lem:Wiener-criterion} by Corollary \ref{lem:Holder-continuous-for-torsion}), we obtain the following H\"older continuous barriers.

\begin{lemma}
	\label{lem:barrier-elliptic-2}
		Let $\Omega\subset X$ be a bounded open domain satisfying a uniform exterior density condition, and let $w\in W^{1,2}(X)\cap C^\alpha(\overline\Omega)$, for some $\alpha\in(0,1)$. Then for any $x_0\in\partial\Omega$ and  $R\in(0,{\rm diam}(\Omega)/100)$,
		 there exist a domain $V_R(x_0)$ with 
		$\Omega_R(x_0)\subset V_R(x_0)\subset \Omega_{11R/10}(x_0),$ 
		and a function $h\in C^\gamma(\overline{V_R(x_0)})\cap W^{1,2}(X)$ for some $\gamma\in(0,\alpha)$ such that 
		$$  {\bf \Delta}h=-1 \ \ {\rm on}\ \ V_R(x_0) $$
		in the sense of distributions, and 
		$$h|_{\partial V_R(x_0)}\gs |w|,\qquad h|_{\partial\Omega\cap B_{R/2}(x_0)}=|w|,\qquad h|_{\partial V_R(x_0)\cap \Omega}= \max_{\overline V_R(x_0)}|w|.$$
	
\end{lemma}

By the comparison principle and the barriers above, we can get the following $L^\infty$ self-improvement property for the solutions to the heat equation near the boundary.

\begin{theorem}
	\label{thm:comparison-heat}
	Let $\Omega\subset X$ be a bounded open domain satisfying the Wiener criterion,   and let $v(x,t)\in  L^2_{\rm loc}\big((0,+\infty); W^{1,2}(\Omega)\big)$ be a sub-solution of the heat equation. $v\gs0$ and is continuous in $\Omega\times(0,+\infty)$. Suppose that $v(x,t)\ls M$ on $\Omega\times(0,+\infty)$.  Let $w\in W^{1,2}(X)\cap C(\overline\Omega)$  and suppose   $\big(v(x,t)-w(x)\big)^+\in W^{1,2}_0(\Omega)$ for almost all $t\in(0,+\infty)$. Take any $x_0\in \partial\Omega$ and any $R\in(0,{\rm diam}(\Omega)/100)$. Let $V_R\subset \Omega$ and the torsion function $h$ be given in Lemma \ref{lem:barrier-elliptic} (or Lemma \ref{lem:barrier-elliptic-2}). 
			 		Then
			  	$$v(x, T)\ls \left(\frac{ M  }{\max_{\overline{V_R(x_0)}}|w|}+1\right)h(x), \qquad \forall x\in  B_R(x_0)\cap \Omega,\quad \forall T> 2\max_{\overline{V_R(x_0)}}|w|.$$
	\end{theorem}

\begin{proof} If $\max_{\overline{V_R(x_0)}}|w|=0$, we have done. So we can assume $\max_{\overline{V_R(x_0)}}|w|>0$. We put $V_R:=V_R(x_0)$ and $L_R:=\max_{\overline{V_R(x_0)}}|w|$ in the following proof. Fix any $T>2L_R$ and define the function
\begin{equation}
\label{equ:barrier-parabolic}\tilde h(x,t):=\frac{h(x)}{L_R}+ \frac{(T-t)^2}{(2L_R)^2}.
\end{equation}
It is a super-solution of the heat equation on $V_R\times (T-2L_R, T)$.	 Indeed, we have that $\tilde h\in W^{1,2}\big(V_R\times (0,T)\big)$ and that for each $t>0$,
	$${\bf \Delta}\tilde h(x,t)=\frac{{\bf \Delta}h}{L_R}= - \frac{\meas}{L_R}	$$ on $V_R$ in the sense of distributions. For each $t\in(T-2L_R,T)$,
	$$\frac{\partial}{\partial t}\tilde h(x,t)= \frac{2(t-T)}{(2L_R)^2}\gs \frac{2(T-2L_R-T)}{(2L_R)^2}=-\frac{1}{L_R}.$$
	It follows for each $t\in(T-2L_R,T)$ that ${\bf \Delta}\tilde h\ls\frac{\partial}{\partial t}\tilde h\cdot \meas$ in the sense of distributions. Therefore, it is a super-solution of the heat equation (see, for example,  \cite[Lemma 6.12]{ZZ18}).

		We first claim that 
		$$v\ls  (M+L_R)\tilde h  \quad {\rm on}\ \  \big(\partial V_{R}  \times (T-2L_R,T)\big)\cup \big(V_R\times\{T-2L_R\}\big)$$  in the sense of Lemma \ref{BRHMHF-lem-3.2}.
		 On $\partial V_R\times (T-2L_R,T)$, we have 
	 $$(M+L_R)\tilde h \gs\frac{(M+L_R)h}{L_R} \gs h\gs|w|\gs   v,$$
	 	 	On $V_R\times(T-2L_R,T-2L_R+\varepsilon)$   for any small $\varepsilon>0$, we have
$$\tilde h(x,t)= \frac{h(x)}{ L_R}+ \frac{(t-T)^2}{(2L_R)^2} \gs 1-\varepsilon/L_R$$
 By the assumption $v\ls M $ on $\Omega\times (0,+\infty)$, we get
	$$(v-a\tilde h)_+\ls a \left(1-(1-\varepsilon/L_R\right),\quad {\rm where}\quad a:=M+L_R,$$
	 which implies the initial condition
	$$\frac 1 \varepsilon\int_{T-2L_R}^{T-2L_R+\varepsilon}\int_{V_R}(v-a\tilde h)_+^2{\rm d}\meas{\rm d}t\to 0\quad {\rm as}\  \ \varepsilon\to0^+.$$  Now we finish the proof of the claim that $v\ls a\tilde h$ near the parabolic boundary of $V_R\times (T-2L_R,T)$ in the sense of Lemma \ref{BRHMHF-lem-3.2}.
	
	By Lemma \ref{BRHMHF-lem-3.2}, we conclude that $v\ls a\tilde h$ almost all in $V_R\times (T-2  L_R,T)$. Since both $v$ and $\tilde h$ are continuous on $U\times(0,T]$, we have
	$$v(x,T)\ls a\tilde h(x, T) =\frac{ah(x)}{L_R} \quad {\rm on} \ \ V_R.$$
 This finishes the proof, since $\Omega\cap B_R(x_0)\subset V_R.$
\end{proof}

As a consequence, we get global continuity for the solutions of the heat equations.
\begin{corollary}
 	\label{cor:global-continuity-heat}
	Let $\Omega\subset X$ satisfy the Wiener criterion,  and let $v(x,t)\in  L^2_{\rm loc}\big((0,+\infty); W^{1,2}(\Omega)\big)$ be a  solution of heat equation with a boundary data $w\in W^{1,2}(X)\cap C(\overline\Omega)$. Then $v(x,t)$ is continuous in $\overline\Omega\times(0,+\infty)$. 	In particular, the Dirichlet heat kernel $p_t^\Omega(x,\cdot)\in C(\overline\Omega)$ for any $x\in\Omega$ and any $t\in(0,+\infty).$	
		\end{corollary}
\begin{proof}
	It is well-known that $v$ is continuous in the interior of $\Omega\times(0,+\infty)$ (see \cite{Stu95}). 
	
	For each $(x_0,T)\in \partial\Omega\times(0,+\infty)$, without loss of generality, we can assume $w(x_0)=0$. Fix a number $R>0$ so small that $T>2\max_{\Omega_{11R/10}(x_0)}|w|$. 
Applying Theorem \ref{thm:comparison-heat} to $v^+$ and $v^-$ ensures the continuity of $v(\cdot,T)$ at $x_0$. 
The proof is finished.
\end{proof}

The same is the case for the boundary $\partial\Omega$ having an exterior density condition and the boundary data being H\"older continuous.
\begin{corollary}
 	\label{cor:global-holder-heat-2}
		Let $\Omega\subset X$ satisfy a uniform exterior density condition, and let $v(x,t)\in  L^2_{\rm loc}\big((0,+\infty); W^{1,2}(\Omega)\big)$ be a  solution of heat equation with a boundary data $w\in W^{1,2}(X)\cap C^\alpha(\overline\Omega)$ for some $\alpha\in(0,1)$. Then, for any $0<t_1<t_2<+\infty$,  $v(x,t)\in C^\gamma(\overline\Omega\times[t_1,t_2])$, for some $\gamma\in(0,\alpha),$ depending on $[t_1,t_2]$. 
		\end{corollary}

\begin{proof}
	The proof is the same as the above, by replacing Lemma \ref{lem:barrier-elliptic} with Lemma \ref{lem:barrier-elliptic-2}.
\end{proof}

\begin{remark}
\begin{itemize}
\item[(1)] The H\"older index $\gamma$ in Lemma \ref{lem:Holder-continuous-for-torsion}, and hence in Lemma \ref{lem:barrier-elliptic-2} and Corollary \ref{cor:global-holder-heat-2} is not optimal. In the next section, we will give a near-optimal estimate for the H\"older index under the assumption that the boundary $\partial\Omega$ has an exterior ball condition.
		\item[(2)] All of the results hold for a geodesic metric measure space $(X,d,\meas)$ supporting a locally doubling measure $\meas$ and a locally $L^2$-Poincar\'e inequality. Here, the assumption that $X$ is a geodesic space has been used in Lemma \ref{lem:good-neighborhood}.
\end{itemize}
\end{remark}

\section{Optimal boundary H\"older continuity}\label{sect4}

  We fix a bounded open subset  $\Omega\subset X$  with $\meas(\partial\Omega)=0$ and ${\rm diam}(\Omega)<{\rm diam}(X)$, in  an $RCD(K,N)$ space $(X,d,\meas)$ with $K\ls0$ and $N\in(1,+\infty)$. 
In this section, we will use $c_1, c_2, C_1, \cdots$ to denote  generic  positive constants depending only on $K$, $N$, $ {\rm diam}(\Omega), \meas(\Omega)$ and $R_{\rm ext}$ in Definition \ref{BRHMHF-def-1.5}; they may differ from line to line.

The following lemma is well known in the Euclidean space $\mathbb R^n$ (see \cite[IV 7.3]{CourantHilbert}, and also \cite{Saf}).

\begin{lemma}
	\label{lem:conrant-hilbert}
	Let $\Omega$ satisfy the exterior ball condition at $x_0\in \partial\Omega$ with a radius $R_{\rm ext}$. Let $v(x)\in C(\overline\Omega)\cap W^{1,2}(\Omega)$ solve 
	$${\bf \Delta}v\gs-1\quad {\rm on}\ \ \Omega$$
	in the sense of distributions. Suppose that $v(x)\gs0$ on $\Omega$  and  $v(x)=0$ on $\partial\Omega$. Then there exists  aconstant $R_1$, depending only on $N,K, {\rm diam}(\Omega)$ and $R_{\rm ext}$, such that for each $R\in(0,R_1)$, 
	$$v(x)\ls  C_1  d(x,x_0)  \cdot \left(\frac{M(R)}{R}+\sqrt{M(R)}\right),\quad\forall\ x\in B_{R}(x_0)\cap \Omega,$$
	where $M(R):= \sup_{B_R(x_0)\cap \Omega}v(x),$ and the constant $C_1$ depends only on $N,K$.
	\end{lemma}

\begin{proof}
  Let $R_1:= \min\{R_{\rm ext},1\}/8.$  For any $R<R_1$, the exterior ball condition gives  a ball $B_R(y_0)$ such that 
$$B_R(y_0)\cap \Omega=\emptyset,\quad  d(y_0,x_0)=R.$$ According the Laplacian comparison theorem on $RCD(K,N)$ space, there exists a number $m_{N,K}>0$ such that 
$${\bf \Delta}d^{-m}_{y_0}\gs \frac{m^2}{R^{m+2}}\cdot \meas,\quad {\rm on}\ X\setminus\{y_0\}$$
for any $m\gs m_{N,K}$ and any $R<1$, in the sense of distributions (see \cite{Gig15} and \cite[Lemma 2.5]{ZZ24}).
Take $$m:= m_{N,K}+\frac{R}{\sqrt{M(R)}}.$$
Here we have assumed $M(R)>0$; otherwise, we are done.

We consider the function 
$$h(x):=A\big(R^{-m}-d_{y_0}^{-m}(x)\big),\quad {\rm where}\quad A:=2R^m\cdot M(R),$$
on $X\setminus\{y_0\}$. Then we have 
$${\bf \Delta}h\ls -A\frac{m^2}{2R^{m+2}}=- \frac{m^2\cdot M(R)}{R^2} \ls -1 $$
on $B_{3R}(y_0)\setminus B_{R/2}(y_0),$ in the sense of distributions, since $m^2 \gs R^2/M(R)$. The exterior ball condition implies 
$h\gs 0=v$ on $\partial \Omega$. A direction computation gives  
$$h\gs A(R^{-m}-(2R)^{-m})\gs AR^{-m}/2=M(R)\gs v\ \ {\rm on}\ \ \partial B_{2R}(y_0)\cap \Omega.$$
Therefore, $h\gs v$ on the boundary of $\Omega\cap B_{2R}(y_0)$. Combining with 
$${\bf \Delta}v\gs -1,\quad {\bf \Delta}h\ls -1,$$
on $\Omega\cap B_{2R}(y_0)$ in the sense of distributions. The maximum principle implies
$$v\ls h,\quad {\rm on}\ \Omega\cap B_{2R}(y_0),$$
which implies the desired estimate, since $h(x_0)=0$,
$$|\nabla h|\ls mAR^{-(m+1)}\ls 2mM(R)/R=2 \left(m_{N,K}\frac{M(R)}{R}+ \sqrt{M(R)}\right) $$
on $\Omega\cap B_{2R}(y_0),$ and $\Omega\cap B_R(x_0) \subset\Omega\cap B_{2R}(y_0).$
\end{proof}

\begin{remark} In the previous section, we essentially only used the assumption that the space $(X,d,\meas)$ supports a locally doubling measure $\meas$ and an $L^2$-Poincar\'e inequality. However, in this section we will work in the $RCD$ setting, since the Laplacian comparison theorem was used in Lemma \ref{lem:conrant-hilbert}.
	\end{remark}

The following regularity of the torsion function will be used in barriers.

\begin{lemma}\label{lem:optimal-regularity-for-torsion}
Let $U\subset \Omega$ satisfy a uniformly exterior ball condition, and   $w\in Lip(\overline U)$. Suppose that  $u\in W^{1,2}(U)$ satifies 
$${\bf\Delta}u=-1,\quad u-w\in W_0^{1,2}(U)$$
 in the sense of distributions. Then $u \in C^{1-\epsilon}(\overline U)$ for any $\epsilon>0.$ 
 \end{lemma}
 \begin{proof}
Let $u_1$ be the harmonic function on $U$ with boundary data $w$, and let $u_2$ be the torsion function with zero boundary data, namely, $u_2$ is the solution of 
$${\bf \Delta}u_2=-1,\quad u|_{\partial\Omega}=0,$$
 in the sense of distributions. The maximum principle implies $u_2\gs0$. Lemma \ref{lem:Wiener-criterion} states $u_2\in C(\overline U)$, in particular, it is bounded on $\overline U$. By using Lemma \ref{lem:conrant-hilbert}, we have 
 \begin{equation}\label{equ:optimal-regularity-torsion-1}
 	|u_2(x)-u_2(x_0)|\ls c_1\cdot d(x,x_0) \cdot \max_{\overline U}u_2
 	 \end{equation}
 for all $x\in U$ and $x_0\in \partial U$, where the constant $c_1$ depends on $N,K,{\rm diam}(\Omega)$ and $ R_{\rm ext}$.   
 
 On the other hand, from \cite[Theorem 1.1]{ZZ24}, we get that $u_1\in C^{1-\epsilon}(\overline U)$ for any $\epsilon>0$. Combining this with (\ref{equ:optimal-regularity-torsion-1}), we conclude that 
 $u\in C^{1-\epsilon}(\overline U).$
The proof is finished.
 \end{proof}

With the help of this regularity, we can get the following optimal boundary H\"older regularity of the heat equations, in the case where the boundary $\partial\Omega$ has an exterior ball condition.\begin{proposition}
	\label{prop:global-holder-heat-3}
		Let $\Omega\subset X$ satisfy a uniform exterior ball condition, and let $v(x,t)\in  L^2_{\rm loc}\big((0,+\infty); W^{1,2}(\Omega)\big)$ be a  solution of heat equation with a boundary data $w\in W^{1,2}(X)\cap Lip(\overline\Omega)$. Then, for any $0<t_1<t_2<+\infty$,  $v(x,t)\in C^{1-\epsilon}(\overline\Omega\times[t_1,t_2])$, for any $\epsilon>0$.  		
\end{proposition}

\begin{proof}
With the help of this regularity of the torsion function in Lemma \ref{lem:optimal-regularity-for-torsion}, combining with a good local neighborhood in Lemma \ref{lem:good-neighborhood}, we can using the same argument as in Lemma \ref{lem:barrier-elliptic} to construct a barrier function $h$  near a boundary point $x_0\in\partial\Omega$ on a local domain $V_R(x_0)$ such that $h$ is $C^{1-\epsilon}(\overline{V_R(x_0)})$. Finally, the assertion comes from Theorem \ref{thm:comparison-heat}, by a same argument in the proof of Corollary \ref{cor:global-continuity-heat}.
\end{proof}

\section{The Zugmund-type boundary estimates for subsolutions}\label{sect-5}
 Let $(X,d,\meas)$ be an $RCD(K,N)$ space  with $K\ls0$ and $N\in(1,+\infty)$. 
 In this section, we continue to fix a bounded open subset  $\Omega\subset X$  with $\meas(\partial\Omega)=0$ and ${\rm diam}(\Omega)< {\rm diam}(X)$. The constants $c_1, c_2, C_1, \cdots$ in this section will depend  only on $K$, $N$, $ {\rm diam}(\Omega), \meas(\Omega)$, $R_{\rm ext}$, and $C_0$ in Definition \ref{BRHMHF-def-1.5}; they may differ from line to line.

\subsection{Boundary gradient estimates for Dirichlet heat kernels}

Let $p^\Omega_t(x,y)$ be the Dirichlet heat kernel on $\Omega$. We first consider its behavior near the boundary.   
\begin{lemma}
	\label{BRHMHF-thm-1.10}
Let $\Omega$ be bounded and satisfy a uniform exterior ball condition with radius $R_{\rm ext}\in(0,{\rm diam}(\Omega)/100)$. Let $  T\gs R^2_{\rm ext}$.  Then  for any $\varepsilon\in(0,T)$ we have
	\begin{equation}\label{BRHMHF-equ-3.12}
			\int_\varepsilon^T|\nabla p^\Omega_t(x,y) |{\rm d}t  \ls  c'_T \cdot \delta(x)  \left(  \int^{{\rm diam}^2(\Omega)}_{d^2(x,y)} \frac{    {\rm d}t }{t\cdot \meas(B_{ \sqrt{t}}(x))}+1\right)
	\end{equation}	
 for any $x\in \Omega$ and  almost all  $y\in \Omega$ such that $d(y,\partial\Omega)\ls \sqrt{\varepsilon}$,
where the constant $ c_T'$ is independent of $\varepsilon$, (they depend only on $T, N,K,\Omega$ and $ R_{\rm ext}$). 
\end{lemma}

 \begin{proof}
  
   It suffices to show  that   (\ref{BRHMHF-equ-3.12})  holds almost eveywhere in  $ B_{\delta(y_0)/2}(y_0)$ for any  $y_0\in \Omega$. We set $R:=\delta(y_0)/2.$

Recall that $p^\Omega_t(x,\cdot)\in Lip_{\rm loc}(\Omega\times (0,T))$ (by \cite[Corollary 1.5]{HZ20}).	Letting $f(\cdot,t):=\ln p^\Omega_t(x,\cdot)$, by local Li-Yau's gradient estimate \cite[Theorem 1.4]{ZZ16}, we have   that
	$$|\nabla f|^2(y,t)-2\frac{\partial f}{\partial t}(y,t)\ls C_{K,N}\left(\frac{1}{R^2}+\frac{1}{t}+1\right)$$
	for almost all $y\in B_R(y_0)$ and almost all $t\in(0,\infty)$. Therefore,  for $t>\varepsilon,$
	$$\frac{|\nabla p^\Omega_t(x,\cdot)|^2(y)}{p^\Omega_t(x,y)}\ls 2\frac{\partial}{\partial t} p^\Omega_t(x,y)+C_{N,K} \left(\frac{1}{R^2}+\frac{1}{\varepsilon}+1\right)p^\Omega_t(x,y)$$	
	for almost all $y\in B_R(y_0)$. Integrating over $[\varepsilon,T]$ and by H\"older inequality, we get
	\begin{equation*}
		\begin{split}
			\left(\int_{\varepsilon}^{T}|\nabla p^\Omega_t(x,\cdot)|(y){\rm d}t\right)^2& \ls  \int_{\varepsilon}^{T}\frac{|\nabla p^\Omega_t(x,\cdot)|^2(y)}{p^\Omega_t(x,y)}{\rm d}t\cdot \int_{\varepsilon}^{T}  p^\Omega_t(x,y){\rm d}t\\
			&\ls \left[2\left(p^\Omega_{T} -p^\Omega_{\varepsilon} \right)+C_{N,K}\left(\frac{1}{R^2}+\frac{1}{\varepsilon}+1\right)\int_{\varepsilon}^{T}  p^\Omega_t(x,y){\rm d}t\right]\cdot \int_{\varepsilon}^{T}  p^\Omega_t(x,y){\rm d}t	\end{split}
	\end{equation*}
for almost all $y\in B_R(y_0)$. By  $p^\Omega_\varepsilon(x,y)\gs0$ and $\varepsilon\gs 4 R^2$ (since $2R=\delta(y_0)\ls \sqrt\varepsilon$), we get
		\begin{equation}\label{BRHMHF-equ-3.15}
		\begin{split}
			\int_{\varepsilon}^{T}|\nabla p^\Omega_t(x,\cdot)|(y){\rm d}t&\ls C_1 \left(R\cdot p^\Omega_{T} (x,y) +  \frac{1}{R } \int_{\varepsilon}^{T}  p^\Omega_t(x,y){\rm d}t\right)   		\end{split}
			\end{equation}	
for almost all $y\in B_R(y_0)$.

By \cite[Theorem 3.9]{ZZ25} and $\delta(y)\ls \delta(y_0)+d(y,y_0)\ls 2\delta(y_0)$, we have 
$$\frac 1 R\int_\varepsilon^Tp^\Omega_t(x,y){\rm d}t\ls \frac{2}{\delta(y_0)}\int_0^\infty p^\Omega_t(x,y){\rm d}t\ls 4 c_3 \cdot \delta(x)  \left(  \int^{D^2}_{d^2(x,y)} \frac{ {\rm d}t }{t\cdot \meas(B_{ \sqrt{t}}(x))}+1\right),$$
where $D:={\rm diam}(\Omega)$.
From (\ref{BRHMHF-equ-1.6}) and $\delta(y)\ls {\rm diam}(\Omega)$, we have
\begin{equation}
	\label{equ:upper-heat-1}
	p^\Omega_T(x,y)\ls \frac{\delta(x) }{T}\frac{c_T\cdot {\rm diam}(\Omega)}{\mu(B_{\sqrt T}(y))}\ls c_T'\cdot \delta(x).
\end{equation}
Substituting the above two inequalities into (\ref{BRHMHF-equ-3.15}), the desired estimate (\ref{BRHMHF-equ-3.12}) follows. The proof is finished.
\end{proof}
 
\begin{remark}
	It is also interesting to obtain a lower bound of $p^\Omega_T(x,y)$ (c.f. \cite{Stu95,Wang98,Zhang02}).  
\end{remark}

\subsection{The set of finite perimeter and one-sided Gauss-Green formula}
The theory of sets of finite perimeter and functions of bounded variation has been generalized to the setting of  $RCD$ spaces (\cite{ABS19, BPS23-jems, BG24}).
\begin{defn}\label{BRHMHF-def-2.7-BV}
A function $f\in L^1(X)$ is called  a function {\emph{of bounded variation}}, denoted by $f\in BV(X)$, if there exist $f_j\in Lip_{\rm loc}(X)$ converging to $f$ in $L^1(X)$ such that
$$\limsup_{j}\int_X|\nabla f_j|{\rm d}\meas<+\infty.$$
\end{defn}

Given a function $ f \in BV(X)$  and an open
set $A\subset X$, one can define
$$|Df|(A):= \inf\left\{ \liminf_{i\to+\infty}\int_A|\nabla f_i|{\rm d}\meas\ \big|\ f_i\in Lip_{\rm loc}(A).\quad f_i\to f\ \ {\rm in}\ \ L_{\rm loc}^1(A)\right\}.
$$

Indeed, (see \cite{ABS19}), the set function $A\mapsto  |Df|(A)$  is the restriction to open sets of a finite Borel measure, which is called the total variation of $f$ and denoted still by $|Df|$. For any Borel set $B$,
$$|Df|(B):= \inf \left\{|Df|(A) \ \big|\  B \subset A,\  A\ {\rm open}\right\}.$$

\begin{defn}\label{BRHMHF-def-2.8}
	A Borel set $E\subset X$ with $\meas(E)<\infty$  is called a {\emph {set of finite perimeter}} if $\chi_E\in BV(X)$. In this case, $|D\chi_E|$ is called the perimeter measure, denoted by $Per_E$.
\end{defn}

In general, a set of finite perimeter can be approximated by level sets of a sequence of Lipschitz functions on $X$. The following fact states a condition for the existence of interior approximations.

 \begin{lemma}\label{BRHMHF-lem5.4}
	Let $\Omega$ be a bounded open set of finite perimeter. Assume that $$\liminf_{r\to0}\frac{\meas\big(B_r(x_0)\setminus\Omega \big)}{\meas\big(B_r(x_0)\big)} \gs \gamma,\quad \forall x_0\in \partial\Omega,$$
  for some $\gamma\in (0,1)$. 	
	 Then there is a sequence $\phi_k\in Lip_0(\Omega)$, $0\ls \phi_k\ls 1$, such that   $\phi_k(x)\to \chi_\Omega$ in $L^1(X)$, as $k\to\infty$,  and
	\begin{equation}
		\label{BRHMHF-equ-4.9}
	\limsup_{k\to\infty}\int_X\eta|\nabla \phi_k|{\rm d}\meas \ls \frac{c}{\gamma}\cdot\int_X\eta {\rm d}Per_{\Omega} , \quad \forall \eta\in C(X), \quad \eta\gs0,
	 \end{equation}
	 	where the constant $c$ depends only on $N,K$ and ${\rm diam}(\Omega)$.	
\end{lemma} 
\begin{proof}
	It was proved in \cite[Lemma 4.12]{ZZ25} that (\ref{BRHMHF-equ-4.9}) holds for any nonnegative function $\eta\in Lip_0(X)$.
	
	For a general $\eta\in C_0(X)$, $\eta\gs0$, given any $\epsilon>0$, there exists a function $\eta_\epsilon\in Lip_0(X)$, $\eta_\epsilon\gs0$,  such that  
	$$\eta_\epsilon\gs \eta-\epsilon,\qquad \int_X\eta_\epsilon {\rm d}Per_\Omega\ls \int_X\eta{\rm d}Per_\Omega+\epsilon.$$
	By applying (\ref{BRHMHF-equ-4.9}) to $\eta_\epsilon+\epsilon$, we conclude that 
	$$\limsup_{k\to\infty}\int_X\eta|\nabla \phi_k|{\rm d}\meas \ls \limsup_{k\to\infty}\int_X(\eta_\epsilon+\epsilon)|\nabla \phi_k|{\rm d}\meas\ls \frac{c}{\gamma}\cdot\left(\int_X\eta {\rm d}Per_{\Omega}+\epsilon+\epsilon Per_\Omega(X)\right).$$	
Letting $\epsilon\to 0^+$, we finished the proof.	
\end{proof}
\begin{remark}
	\label{rem:rem5.5-BRHMHF} In the Euclidean setting, some conditions for the interior approximations of a set of finite perimeter have been given by Chen-Li-Torres \cite{CLT20}, Comi-Torres \cite{CT17}, and Gui-Hu-Li \cite{GHL23}.
\end{remark}

\subsection{Boundary estimates for nonnegative subsolutions}
Let us begin with the following Duhamel-type principle.
\begin{theorem}\label{BRHMHF-thm-5.3} Let $\{\phi_k\}_{k\in\mathbb N}$ be a sequence of functions such that  $\phi_k\in Lip_0(\Omega)$, $0\ls \phi_k\ls 1$ for each $k$, and $ \phi_k\to \chi_\Omega$ in $L^1(\Omega)$ as $k\to\infty.$
 Let $f\in W^{1,2}(\Omega\times(0,T_0))\cap L^\infty(\Omega\times(0,T_0))$, $f\gs0$, such that  $\partial_tf(x,t)\in L^\infty_{\rm loc}(\Omega\times(0,T_0))$. Assume that 
 $${\bf \Delta}f(x,t)\gs \partial_tf(x,t)+a(x,t)$$
 for some $a(x,t)\in L^1(\Omega\times(0,T_0))$ in the sense of distributions.
   Then for any $0<\varepsilon<T<T_0$ , we have for all $x\in\Omega$ that
\begin{equation}\label{equ:Duhamel-1}
		\begin{split}
	&  \int_\Omega p_\varepsilon^\Omega(x,y)f(y,T)\dm(y)-\int_\Omega p_T^\Omega(x,y)f(y, \varepsilon )\dm(y)\\
	\ls & \liminf_{k\to+\infty}\int_\Omega\int_\varepsilon^T|\nabla p_t^\Omega(x,\cdot)|f(\cdot,T+\varepsilon -t){\rm d}t \cdot |\nabla\phi_k|\dm- \int_\Omega\int_\varepsilon^T  p^\Omega_t(x,\cdot)a(\cdot,T+\varepsilon-t){\rm d}t\dm.
	\end{split}
	 	\end{equation}	
\end{theorem}
\begin{proof}Fix $0<\varepsilon<T<T_0$, $x\in\Omega$, and set
	$$J_k(x,t):=\int_\Omega\phi_k (y)p^\Omega_t(x,y)f(y,T+\varepsilon-t)\dm(y).$$
	From \cite[Theorem 1.1]{HZ20}, we have $\partial_tp^\Omega_t(x,\cdot)\in L^\infty_{\rm loc}(\Omega\times(0,T_0))$. Combining with the assumption $\partial_tf(x,t)\in L^\infty_{\rm loc}(\Omega\times(0,T_0))$, the function $t\mapsto J_k(x,t)$ is locally Lipschitz continuous in $(0,T_0)$, and for almost all $t\in(0,T_0)$ that 
	\begin{equation*}
		\begin{split}\partial_t J_k(x,t)&=\int_\Omega\phi_k(y)\Big(\partial_tp^\Omega_t(x,y)f(y,T+\varepsilon -t)+p^\Omega_t(x,y)\partial_tf(y,T+\varepsilon -t)\Big)\dm(y)\\
		&= \int_\Omega\phi_k f(\cdot,T+\varepsilon -t)\Delta p^\Omega_t(x,\cdot)\dm-\int_\Omega\phi_k p^\Omega_t(x,\cdot)\partial_t f(\cdot,T+\varepsilon -t)\dm \\
		&=-\int_\Omega \phi_k\ip{\nabla f(\cdot,T+\varepsilon -t)}{\nabla p_t^\Omega(x,\cdot)}\dm- \int_\Omega\ip{\nabla p_t^\Omega(x,\cdot)}{\nabla\phi_k}f(\cdot,T+\varepsilon -t)\dm \\
		&\quad -\int_\Omega\phi_k p^\Omega_t(x,\cdot) \partial_tf(\cdot,T+\varepsilon -t)\dm 
					\end{split}
	\end{equation*}
	where we have used  $\phi_k f(\cdot,T+\varepsilon -t)\in W^{1,2}(\Omega)$. Integrating over $(\varepsilon,T)$ gives 
	\begin{equation*}
		\begin{split}
	&\quad	J_k(x,T)-J_k(x,\varepsilon)\\
	&\gs-\int_\varepsilon^T\int_\Omega \phi_k\ip{\nabla f(\cdot,T+\varepsilon -t)}{\nabla p_t^\Omega(x,\cdot)}\dm{\rm d}t- \int_\varepsilon^T\int_\Omega|\nabla p_t^\Omega(x,\cdot)|\cdot |\nabla\phi_k|f(\cdot,T+\varepsilon -t) \dm{\rm d}t \\
		&\quad -\int_\varepsilon^T\int_\Omega\phi_k p^\Omega_t(x,\cdot) \partial_tf(\cdot,T+\varepsilon -t)\dm{\rm d}t,
					\end{split}
	\end{equation*}	
	where we have used $f\gs0$ and $-\ip{\nabla p_t^\Omega}{\nabla \phi_k}\gs -|\nabla p^\Omega
	_t|\cdot |\nabla \phi_k|.$
	
Notice that $p_t^\Omega(x,\cdot)f(\cdot,t)\in L^1(\Omega)$ for any $t\in[\varepsilon,T]$, and $\phi_k\in L^\infty(\Omega)$, the donimated convergence theorem implies
$$\lim_{k\to+\infty}J_k(x,t)= 	\int_\Omega p_t^\Omega(x,y)f(y,T+\varepsilon-t)\dm(y),\qquad \forall t\in[\varepsilon,T] .$$
Similarly, by $\ip{\nabla f(\cdot,T+\varepsilon -t)}{\nabla p_t^\Omega(x,\cdot)}\in L^1(\Omega\times(\varepsilon,T))$ and $p^\Omega_t(x,\cdot) \partial_tf(\cdot,T+\varepsilon -t)\in L^1(\Omega\times(\varepsilon,T))$, and letting $k\to+\infty$, we obtain
	\begin{equation*}
		\begin{split}
	&\quad	\int_\Omega p_T^\Omega(x,y)f(y, \varepsilon )\dm(y)-\int_\Omega p_\varepsilon^\Omega(x,y)f(y,T )\dm(y)\\
	&\gs-\int_\varepsilon^T\int_\Omega  \ip{\nabla f(\cdot,T+\varepsilon -t)}{\nabla p_t^\Omega(x,\cdot)}\dm{\rm d}t -\int_\varepsilon^T\int_\Omega  p^\Omega_t(x,\cdot) \partial_tf(\cdot,T+\varepsilon -t)\dm{\rm d}t\\
	&\quad -\liminf_{k\to+\infty}\int_\Omega\int_\varepsilon^T|\nabla p_t^\Omega(x,\cdot)|f(\cdot,T+\varepsilon -t){\rm d}t \cdot |\nabla\phi_k|\dm.
	\end{split}
	\end{equation*}		
By ${\bf\Delta}f\gs\partial_tf+a$ in the sense of distributions and $p_t^\Omega(x,\cdot)\in W^{1,2}_0(\Omega)\cap L^\infty(\Omega)$, we have 
\begin{equation*}
\begin{split}
	-\int_\Omega p^\Omega_t(x,\cdot)\partial_tf(\cdot, T+\varepsilon-t)\dm &\gs  \int_\Omega p^\Omega_t(x,\cdot) \Big(-{\rm d}{\bf\Delta} f(\cdot, T+\varepsilon-t)+a(\cdot,T+\varepsilon-t)\dm\Big)\\
	&=\int_\Omega \ip{\nabla p^\Omega_t(x,\cdot)}{\nabla f(\cdot, T+\varepsilon-t)} \dm+p^\Omega_t(x,\cdot)a(\cdot,T+\varepsilon-t)\dm\end{split}	
\end{equation*}
	Adding the above inequalities implies the desired assertion.
\end{proof}

\begin{theorem}\label{BRHMHF-thm-5.4} Let $\Omega$ be a bounded open set of finite perimeter, and suppose that $\Omega$ satisfies a uniformly exterior ball condition with $R_{\rm ext}>0$. 
 Let $f\in W^{1,2}(\Omega\times(0,+\infty))$, $0\ls f(x,t)\ls M$ for some $M>0$ on $\Omega\times(0,+\infty)$,   $f\in C(\overline\Omega\times[a,b])$ for any $[a,b]\subset (0,+\infty)$ and $f(x,t)=w(x)$ for all $x\in\partial\Omega$ and any $t\in(0,+\infty)$. Suppose $\partial_tf\in L^\infty_{\rm loc}(\Omega\times(0,+\infty))$. Assume that 
 $${\bf \Delta}f(x,t)\gs \partial_t f(x,t)$$
on $\Omega\times(0,+\infty)$ in the sense of distributions. Take any $x_0\in \partial\Omega$ and any $R\in(0,{\rm diam}(\Omega)/100)$. 
   Then for any $T\gs \max\{R^2_{\rm ext},2\max_{\Omega\cap B_{2R}(x_0)}|w|\}$ we have   
    $$\frac{f(x,T+T_1)}{\delta(x)}\ls c_{R,T}\left( \int_\Omega f(y,T_1)\dm(y)+ \int_{\partial\Omega}\left(\int_{d^2(x,y)}^{D^2}\frac{{\rm d}t}{t\cdot \meas(B_{\sqrt t}(x))}+1\right)|w(y)|{\rm d}Per_\Omega(y)+1\right)$$
for any $x\in\Omega\cap B_{R/4}(x_0)$ and any $T_1>0$,    where $D:={\rm diam}(\Omega)$.
   \end{theorem}

\begin{proof}
Fix any $T\gs \max\{R^2_{\rm ext},2\max_{\Omega\cap B_{2R}(x_0)}|w|\}$. Since $\Omega$ is a set of finite perimeter and it satisfies a uniformly exterior ball condition with $R_{\rm ext}$, according to Lemma \ref{BRHMHF-lem5.4}, there exists a family of functions $\phi_k\in Lip_0(\Omega)$, $0\ls \phi_k\ls 1$, $\phi_k\to \chi_\Omega$ in $L^1(\Omega)$, and 
\begin{equation}\label{equ:BRHMHF-5.6}
	\limsup_{k\to\infty}\int_X\eta|\nabla \phi_k|\dm\ls c_1\int_X\eta {\rm d}Per_\Omega,\quad \forall\eta\in C_0(X),\quad \eta\gs0.
\end{equation} 
Theorem \ref{BRHMHF-thm-5.3} states for any $\varepsilon\in(0,T)$ that 
 \begin{equation}\label{equ:Duhamel-2}
 		\begin{split}
 	 & \int_\Omega p_\varepsilon^\Omega(x,y)f(y,T_1+T)\dm(y)-\int_\Omega p_T^\Omega(x,y)f(y, T_1+\varepsilon)\dm(y) \\	\ls & \liminf_{k\to+\infty}\int_\Omega\int_\varepsilon^T|\nabla p_t^\Omega(x,\cdot)|f(\cdot,T_1+T+\varepsilon -t){\rm d}t \cdot |\nabla\phi_k|\dm.
  	\end{split}
 	 	\end{equation}

 From Theorem \ref{thm:comparison-heat} and assuming $\varepsilon<R^2$, we have for any $t\in(0,T)$ that 
\begin{equation*}
	f(y,T_1+T+\varepsilon-t) \ls \eta(y):=\begin{cases}
		\big(M/(\max_{\Omega\cap V_{R}(x_0)}|w|)+1\big)h(y), & y\in V_{R}(x_0)\cap \Omega,\\
		M+\max_{\Omega\cap V_{R}(x_0)}|w|,& 	\Omega\setminus V_{R}(x_0),
	\end{cases}  
\end{equation*}
  where $V_R$ and  $h(y)$ are given in Theorem \ref{thm:comparison-heat}.
   Now, by combining this with (\ref{BRHMHF-equ-3.12}), we have
\begin{equation*}
\int_\varepsilon^T|\nabla p^\Omega_t(x,y)|f(y,T_1+T+\varepsilon-t){\rm d}t\ls \eta(y)\int_\varepsilon^T|\nabla p^\Omega_t(x,y)| {\rm d}t\ls c_T'\delta(x)\cdot F_x(y)\eta(y),
 \end{equation*}
 for all $y\in\Omega$ such that $\delta(y)<\sqrt\varepsilon,$
 where $$F_x(y):= \int_{d^2(x,y)}^{D^2}\frac{{\rm d}t}{t\cdot \meas(B_{\sqrt t}(x))}+1.$$
By (\ref{equ:BRHMHF-5.6}), we have
\begin{equation}
	\label{equ:BRHMHF-5.8}
\limsup_{k\to+\infty}\int_\Omega\int_\varepsilon^T|\nabla p^\Omega_t(x,y)|f(y,T_1+T+\varepsilon-t)  |\nabla\phi_k|{\rm d}t\dm	\ls c_1c'_T\delta(x)\int_{\partial\Omega}F_x(y)\eta(y){\rm d}Per_\Omega(y).
\end{equation}	
Since $\eta(y)=|w(y)|$ on $\partial\Omega\cap B_{R/2}(x_0)$ and $\eta(y)\ls M+\max_{\Omega\cap V_R(x_0)}|w|$ on $\partial\Omega\setminus B_{R/2}(x_0)$, we get
\begin{equation*}
	\begin{split}
		\int_{\partial\Omega}F_x(y)\eta(y){\rm d}Per_\Omega(y)&\ls A\int_{\partial\Omega\cap B_{R/2}(x_0)}F_x(y)|w(y)|{\rm d}Per_\Omega(y)+B \int_{\partial\Omega\setminus B_{R/2}(x_0)}F_x(y) {\rm d}Per_\Omega(y)\\
		&\ls A\int_{\partial\Omega}F_x(y)|w(y)|{\rm d}Per_\Omega(y)+B \cdot c_R\cdot Per_\Omega(y)(X),
		\end{split}
\end{equation*}
where 
$$ A:=M/(\max_{\Omega\cap V_{R}(x_0)}|w|)+1,\qquad B:= M + \max_{\Omega\cap V_{R}(x_0)}|w|,$$
and we have used $F_x(y)\ls c_R$ on $\partial\Omega\setminus B_{R/2}(x_0)$, since $d(x,y)\gs R/4$ (by $x\in B_{R/4}(x_0)$). 
 The combination of this and (\ref{equ:Duhamel-2})-(\ref{equ:BRHMHF-5.8}) implies 
$$    \int_\Omega p_\varepsilon^\Omega(x,y)f(y,T_1+T)\dm(y)-\int_\Omega p_T^\Omega(x,y)f(y, T_1+\varepsilon)\dm(y) 	\ls   A_1 \delta(x)\int_{\partial\Omega}F_x(y)|w(y)|{\rm d}Per_\Omega(y)+B_1\cdot \delta(x),$$
where $A_1:=c_1c'_TA$ and $B_1:=c_1c_T' B c_RPer_\Omega(X).$ 
Letting $\varepsilon\to0$ and using the upper bound of heat kernel (\ref{equ:upper-heat-1}), we obtain
$$f(x,T_1+T)\ls c'_{R,T}\delta(x)\left(\int_\Omega f(y,T_1)\dm(y)+ \int_{\partial\Omega}F_x(y)|w(y)|{\rm d}Per_\Omega(y)+1\right),$$	
which is the desired estimate.
\end{proof}

As a consequence, we have the following Zygmund-type estimate for sub-solutions of heat equations near the boundary.

\begin{corollary}
	\label{cor:coro5.7-BRHMHF} Let $\Omega$ be a bounded open set of finite perimeter. Suppose that $\Omega$ satisfies a uniform exterior ball condition with $R_{\rm ext}>0$, and that $\partial\Omega$ is perimetrically regular with constant $C_0$. Let $f$ and $w$ be given in Theorem \ref{BRHMHF-thm-5.3}. Suppose that $w\in Lip(\overline\Omega)$ with a Lipschitz constant $L_w>0$. Let $x_0\in \partial\Omega$ and suppose $w(x_0)=0$.
    Then for any $T\gs R^2_{\rm ext}$  we have    
    $$\frac{f(x,T_1+T)}{\delta(x)}\ls c_T \ln\left(\frac{e\cdot {\rm diam}(\Omega)}{\delta(x)}\right)$$
for all $x\in\Omega\cap B_{R/4}(x_0)$ with $d(x,x_0)\ls 2\delta(x)$ and all $T_1>0$, where $R:=\min\{{\rm diam}(\Omega)/100,  R^2_{\rm ext}/(4L_w)\}$ and the constant $c_T$ depends only on $N,K,\Omega,M, L_w$ and $T$.
   \end{corollary}

\begin{proof}
	The fact $w(x_0)=0$ and $w\in Lip(\partial\Omega)$ with Lipschitz constant $L_w$ implies
	$$|w(x)|\ls L_wd(x,x_0),\qquad \forall x\in\Omega,$$
Since $$\max_{B_{2R}(x_0)\cap \Omega}|w|\ls 2RL_w\ls R^2_{\rm ext}/2.$$
	by combining this with \cite[Lemma 5.4]{ZZ25} and Theorem \ref{BRHMHF-thm-5.3}, we get 
 $$\frac{f(x,T_1+T)}{\delta(x)}\ls c_{R,T} \int_\Omega f(y,T_1)\dm(y)+c_{R,T} \left[C\ln\left(\frac{{\rm diam}(\Omega)}{\delta(x)}\right)+\int_{\partial\Omega} |w(y)|{\rm d}Per_\Omega(y)+1\right]$$
for all $x\in\Omega\cap B_{R/4}(x_0)$ with $d(x,x_0)\ls 2\delta(x)$ and $T_1>0$. Notice $f(y,T_1)\ls M$.
Therefore, we conclude that 
$$\frac{f(x,T_1+T)}{\delta(x)}\ls c_{R,T}\left[ \meas(\Omega)\cdot M+ C\ln\left(\frac{{\rm diam}(\Omega)}{\delta(x)}\right)+Per_\Omega(X)\cdot L_w\cdot {\rm diam}(\Omega)+1 \right]$$
for all $x\in\Omega\cap B_{R/4}(x_0)$ with $d(x,x_0)\ls 2\delta(x)$ and $T_1>0$.	
\end{proof}

 \section{Proof of Theorem \ref{thm:main-BRHMHF}.}\label{sect-6}
 
 This section will deal with the harmonic map heat flow into a $CAT(0)$ space $(Y,d_Y)$.  Let $w\in W^{1,2}(\Omega,Y)\cap C(\overline\Omega)$ and let $u(x,t)$ be a harmonic map heat flow with initial data $u_0\in W^{1,2}(\Omega,Y)\cap L^\infty(\Omega,Y)$ and boundary data $w$.   

We need the following facts in \cite{HZZ26} 
\begin{lemma}\label{lem:BRHMHF-lem6.1}
Suppose that $d_Y(P_0,u_0(x))\ls M$ for some $P_0\in Y$ and $M>0$. Then the following hold:
\begin{itemize}
\item[(1)] $d_Y(P_0,u(x,t))\ls M$ for any $t>0$,
\item[(2)] $d_Y(Q,u(x,t))$ is a sub-solution of the heat equation on $\Omega\times(0,+\infty)$, for any $Q\in Y$, and 
\item[(3)] $u$ is loally continuous on $\Omega\times(0,+\infty)$.
\end{itemize}
\end{lemma}

We are now in a position to prove Theorem \ref{thm:main-BRHMHF}.
\begin{proof}[Proof of Theorem \ref{thm:main-BRHMHF}]
Take any $x_0\in \partial\Omega$ and let $w_{x_0}(x):=d_Y(w(x),w(x_0)).$ It is clear that if $w\in C(\overline\Omega,Y)$ (or $C^\gamma(\overline\Omega,Y)$, $Lip(\overline\Omega,Y)$) then $w_{x_0}\in C(\overline\Omega)$ (resp. $C^\gamma(\overline\Omega)$, $Lip(\overline\Omega)$). Moreover, the modulus of continuity of $w_{x_0}$  (or the H\"older norm, the Lipschitz constant $L_{w_{x_0}}$) is not great than the modulus of continuity of $w$ (resp. the H\"older norm, the Lipschitz constant $L_{w}$).

 The triangle inequality implies
$$d_Y(u(x,t),w(x_0))-w_{x_0}(x)\ls d_Y(u(x,t),w(x))$$
for any $(x,t)\in\Omega\times(0,+\infty)$. Hence $\big(d_Y(u(x,t),w(x_0))\big)^+\in W^{1,2}_0(\Omega)$ for almost all $t\in(0,+\infty)$.

The items (1), (2) and (3) follow from the combination of Lemma \ref{lem:BRHMHF-lem6.1}, Theorem \ref{thm:comparison-heat}, and the regularity of the torsion function $h$ in Lemma \ref{lem:barrier-elliptic}, Lemma \ref{lem:barrier-elliptic-2} and Lemma \ref{lem:optimal-regularity-for-torsion}, respectively.

Let  $R:=\min\{{\rm diam}(\Omega)/100,  R^2_{\rm ext}/(4L_w)\}$. For any $x\in \Omega\cap B_{R/4}(x_0)$, since the function 
$$(x,t)\mapsto d_Y(u(x,t),w(x_0))$$
is a sub-solution of the heat equation with the boundary data $w_{x_0}(x)$. The triangle inequality yields $\partial_t d_Y(u(x,t),w(x_0))\in L^\infty_{\rm loc}(\Omega\times(0,+\infty)$ (because $u$ is locally Lipschitz continuous in time). Corollary \ref{cor:coro5.7-BRHMHF} implies
$$\frac{d_Y(u(x,T+T_1),w(x_0))}{\delta(x)}\ls  c_T \ln\left(\frac{e\cdot {\rm diam}(\Omega)}{\delta(x)}\right)$$
for any  $x\in\Omega\cap B_{R/4}(x_0)$ with $d(x,x_0)\ls 2\delta(x)$ and any $T_1>0$.
\end{proof}

\end{document}